\documentclass[preprint,12pt]{elsarticle}
\usepackage{amsmath,amssymb,amsthm,url}

\newtheorem{thr}{Theorem}
\newtheorem{lem}[thr]{Lemma}
\newtheorem{prop}[thr]{Proposition}
\newtheorem{cor}[thr]{Corollary}

\newtheorem{claim}[thr]{Claim}

\theoremstyle{definition}

\theoremstyle{remark}
\newtheorem{remr}[thr]{Remark}

\def\R{\mathbb{R}}

\def\ic{\operatorname{ic}}
\def\trd{\operatorname{trdeg}}
\def\col{\operatorname{col}}
\def\pout{\mathcal{P}_{out}}
\def\pin{\mathcal{P}_{in}}

\journal{arXiv.org}

\begin{document}

\begin{frontmatter}













\title{Euclidean distance matrices and separations in communication complexity theory}


\author{Yaroslav Shitov}

\ead{yaroslav-shitov@yandex.ru}

\address{National Research University Higher School of Economics, 20 Myasnitskaya Ulitsa, Moscow 101000, Russia}

\begin{abstract}
A Euclidean distance matrix $D(\alpha)$ is defined by $D_{ij}=(\alpha_i-\alpha_j)^2$, where $\alpha=(\alpha_1,\ldots,\alpha_n)$ is a real vector. We prove that $D(\alpha)$ cannot be written as a sum of $\left[2\sqrt{n}-2\right]$ nonnegative rank-one matrices, provided that the coordinates of $\alpha$ are algebraically independent. This result allows one to solve several open problems in computation theory. In particular, we provide an asymptotically optimal separation between the complexities of quantum and classical communication protocols computing a matrix in expectation.
\end{abstract}


\end{frontmatter}

\section{Introduction}

A significant number of recent publications are devoted to the study of different rank functions of matrices arising from different measures of complexity in the theory of computation. Examples of such functions include the nonnegative and positive semidefinite ranks of a matrix, the quantum and classical communication complexities and many others. The aim of our paper is to solve several open problems concerning the mutual behaviour of these functions.

Let $A$ be a real matrix with nonnegative entries. The \textit{nonnegative rank} of $A$ is the smallest integer $k$ such that $A$ can be written as a sum of $k$ rank-one nonnegative matrices. The nonnegative rank arises in the theory of computation as the measure of complexity of a linear program describing a polytope corresponding to a given matrix~\cite{Yan}. Another interesting rank function, known as the \textit{positive semidefinite} (or \textit{psd}) rank, arises in the similar fashion but from semidefinite descriptions of polytopes~\cite{psd-rank}. More precisely, the psd rank of $A$ is the smallest $k$ such that there are two tuples of positive semidefinite $k\times k$ matrices, $(B_1,\ldots,B_n)$ and $(C_1,\ldots,C_m)$, such that an $(i,j)$th entry of $A$ equals $\operatorname{tr}(B_iC_j)$.

Also, the functions introduced above have applications in the communication complexity theory. For instance, the value $\lceil\log_2\operatorname{rank}_+(A)\rceil$ is the optimal size of a \textit{classical randomized} communication protocol computing $A$ in expectation. Similarly, $\lceil\log_2\operatorname{rank}_{psd}(A)\rceil$ is the optimal size of a \textit{quantum} communication protocol computing $A$. We refer the reader to~\cite{Jain} for a more detailed treatment of these questions. We also note that the above mentioned rank functions find several applications not directly related to computer science. In particular, the concept of nonnegative rank is important in statistics~\cite{KRS}, data mining~\cite{LS} and many other contexts~\cite{CR}.

\section{Our results}

Our paper deals with the family of so called \textit{Euclidean distance matrices}, which are an interesting source of examples illustrating the behavior of the above mentioned functions. Let $\alpha=(\alpha_1,\ldots,\alpha_n)$ be a real vector with $n\geqslant3$ and pairwise distinct coordinates. We define the Euclidean distance matrix as the $n\times n$ matrix $D(\alpha)$ whose $(i,j)$th entry equals $D_{ij}=(\alpha_i-\alpha_j)^2$. Beasley and Laffey~\cite{BL} showed that the classical rank of the matrix $D(1,2,\ldots,n)$ equals three and that the nonnegative rank of it gets arbitrarily large as $n$ goes to infinity. They conjectured that the maximal possible rank of an $n\times n$ Euclidean 
distance matrix is $n$, but this conjecture has been refuted in~\cite{my7x7}. In the abstract of~\cite{Hrubes}, Hrube\v{s} mentions the problem asking whether or not the condition $\operatorname{rank}_+D(\alpha)\in O(\ln n)$ holds for all $\alpha\in\R^n$. He gives an affirmative solution of this problem for some families of vectors $\alpha\in\R^n$, but the general case remains open. As we will show, the solution of this problem is in fact negative. Actually, we will prove that almost all vectors $\alpha\in\mathbb{R}^n$ are such that $\operatorname{rank}_+D(\alpha)$ grows as a power of $n$.

\begin{thr}\label{thrdistlower}
If the coordinates of a vector $\alpha\in\mathbb{R}^n$ are algebraically independent over $\mathbb{Q}$, then $\operatorname{rank}_+ D(\alpha)\geqslant 2\sqrt{n}-2$.
\end{thr}

\begin{remr}
One can construct a family of $n$ algebraically independent numbers as $b_i=\exp a_i$, where $a_1,...,a_n$ is a family of real numbers linearly independent over $\mathbb{Q}$. (This is the famous \textit{Lindemann--Weierstrass theorem}.) Therefore, the lower bound for the nonnegative rank as in the above theorem works for $D(b_1,\ldots,b_n)$ as well.
\end{remr}

We can get some other interesting separations as corollaries of Theorem~\ref{thrdistlower}. In particular, let us compare the behaviour of the nonnegative and psd ranks. It is a basic result that the former is greater than or equal to the latter, but how large the difference can be? As pointed out in~\cite{psd-rank}, the psd rank of $D(\alpha)$ is always equal to two. Therefore, the above mentioned result from~\cite{BL} yields an example of a family of matrices whose psd-ranks are bounded but nonnegative ranks grow logarithmically with the size of a matrix. The foundational paper~\cite{FMPTdW} provides a family of $n\times n$ matrices whose nonnegative ranks grow as a power of $n$ while the psd ranks grow logarithmically. The question of whether this separation is optimal has been left open. In particular, do there exist matrices with bounded psd ranks whose nonnegative ranks grow as a power of the size? The problem of separating the nonnegative and psd ranks has been discussed also in~\cite{Jain}, but the above mentioned question remained open. We get the answer as a corollary of Theorem~\ref{thrdistlower}.

\begin{cor}\label{cordistlower1}
There is a matrix $D\in\R^{2^n\times2^n}$ such that $\operatorname{rank}_{psd}(D)=2$ and $\operatorname{rank}_{+}(D)=2^{\Omega(n)}$.
\end{cor}

This corollary is also interesting from the point of view of the communication complexity theory. As pointed out above, the logarithms of psd and nonnegative ranks, respectively, are optimal sizes of quantum and classical communication protocols computing a given matrix in expectation. Therefore, we get the asymptotically optimal separation between the quantum and classical communication complexities. The existence of such a separation was an open problem despite the efforts mentioned in the above paragraph. The corresponding question was explicitly posed in~\cite{Zhang} as Problem~4 in Section~5.

\begin{cor}\label{cordistlower2}
There is a nonnegative matrix $D\in\R^{2^n\times2^n}$ which can be computed with a one-bit quantum communication protocol but requires $\Omega(n)$ bits to be computed by a classical randomized protocol in expectation.
\end{cor}

The goal of this paper is to prove Theorem~\ref{thrdistlower}. Our approach is mostly geometric, and we use the characterization of the nonnegative rank in terms of the classical \textit{nested polytopes} problem. The necessary general results and a description of our technique are provided in Section~\ref{sectech}. The proof of Theorem~\ref{thrdistlower} is completed in Section~\ref{secpr}.

\section{Our technique}\label{sectech}

The foundational paper~\cite{Yan} by Yannakakis established a connection between the nonnegative rank and the concept known as the \textit{extension complexity} of a polytope. Our proof of Theorem~\ref{thrdistlower} is based on the variation of Yannakakis' theory which we develop in this section.

Instead of working with a more common concept of extension complexity, we work with a somewhat dual concept of \textit{intersection complexity}, see~\cite{PP} for details. These invariants are always equal to each other, and all the results on one of them hold for the other as well. (However, we do not use the fact that they are equal in the proof of Theorem~\ref{thrdistlower}.) Let $P\subset\R^d$, $Q\subset\R^n$ be polytopes and $H=\{x\in\R^n\left| x_{d+1}=\ldots=x_n=0\right.\}$ a plane in $\R^n$. We say that $P$ is a \textit{slice} of $Q$ if $P=Q\cap H$. The \textit{intersection complexity} of $P$, denoted $\ic(P)$, is the smallest integer $k$ such that $P$ is a slice of a polytope with $k$ vertices.

We will say that a real vector is \textit{stochastic} if its entries sum up to one. Now let $A$ be a nonnegative column-stochastic $n\times m$ matrix. That is, we assume that the columns of $A$ are taken from the standard simplex $\Delta_n=\{x\in\R^n\left|x_1+\ldots+x_n=1,x_i\geqslant0\right.\}$. We denote by $\operatorname{col}(A)$ the linear subspace of $\R^n$ spanned by the columns of $A$, and we define the two polytopes, $\mathcal{P}_{in}(A)$ and $\mathcal{P}_{out}(A)$, as follows. We set $\mathcal{P}_{out}(A)=\Delta_n\cap\col(A)$, and we define $\mathcal{P}_{in}(A)$ as the convex hull of the vertices of $A$. The following proposition can be seen as a variation of the study of nested polytope problems by Gillis and Glineur~\cite{GG}.

\begin{prop}\label{proppinpout}
Let $A$ be a nonnegative column-stochastic $n\times m$ matrix. If $\operatorname{rank}_+(A)\leqslant r$, then there is a polytope $P$ satisfying $\mathcal{P}_{in}(A)\subset P\subset \mathcal{P}_{out}(A)$ and $\ic(P)\leqslant r$.
\end{prop}

\begin{proof}
Let $A=BC$ be a nonnegative factorization in which $B$ has at most $r$ columns. Since the transformation $(B,C)\to(BD,D^{-1}C)$ does not change the product $BC$, we can perform the scaling of the columns of $B$ and the corresponding scaling of the rows of $C$. Therefore, we can assume without loss of generality that the columns of $B$ belong to $\Delta_n$. We denote by $R$ the intersection of the convex hull of the columns of $B$ with the affine subspace $\col(A)\cap\{x_1+\ldots+x_n=1\}$. Since $B$ is nonnegative, we have $R\subset\pout(A)$; we also have $\pin(A)\subset R$ because the columns of $A$ are nonnegative combinations (and, therefore, convex combinations) of the columns of $B$.
\end{proof}

Now we are going to prove a useful lower bound on the intersection complexity of a polytope. Let $P$ be a polytope in $\mathbb{R}^d$; we denote by $\mathbb{Q}(P)$ the field obtained from $\mathbb{Q}$ by adjoining the coordinates of vertices of $P$. By  $\trd(P)$ we denote the transcendence degree of the field extension $\mathbb{Q}(P)\supset\mathbb{Q}$. The following results provide a lower bound for the quantity $\ic(P)$ in terms of $\trd(P)$. These results can be seen as an algebraic analogue of the corresponding result in~\cite{Padrol}, which itself is a generalization of the result in~\cite{FRT}.

\begin{lem}\label{lemlowtrdeg}
Let $Q\subset\R^d$ be a polytope with $v$ vertices, $l$ a rational affine subspace of $\R^d$, and $P=Q\cap l$. If $\dim Q=d$, $\dim l=k$, then $\trd(P)\leqslant d(v-d+k)$.
\end{lem}

\begin{proof}
Let $U=(u_1,\ldots,u_{k+1})$ be a tuple of arbitrary points on $l$ satisfying $\dim\operatorname{conv} U=k$. We can find a tuple $V=(v_1,\ldots,v_{d-k})$ of $d-k$ vertices of $Q$ satisfying $\dim\operatorname{conv}\,U\cup V=d$.

Let $V=(v_1',\ldots,v_{d-k}')$ be a tuple of arbitrary rational points satisfying $\dim\operatorname{conv}\,U\cup V'=d$. Then there exists a unique affine transformation $\pi$ sending $(U,V)$ to $(U,V')$. Clearly, $\pi$ is identical on $l$, and the polytope $\pi(Q)$ has $d-k$ vertices with rational coordinates. We get $\trd(P)\leqslant\trd(\pi(Q))\leqslant dv-d(d-k)$.
\end{proof}

\begin{thr}\label{thrlowtrdeg}
Let $P\subset\R^k$ be a polytope. Then $\ic(P)\geqslant2\sqrt{\trd(P)}-k$.
\end{thr}

\begin{proof}
Assume that there is a $d$-dimensional polytope $Q$ with $v$ vertices such that $P$ is a slice of $Q$. By Lemma~\ref{lemlowtrdeg}, we get $\trd(P)\leqslant d(v+k-d)$. The expression $d(v+k-d)$ attains its maximum at $d=(v+k)/2$, so we get $4\trd(P)\leqslant(v+k)^2$ or $v\geqslant2\sqrt{\trd(P)}-k$. 
\end{proof}

Recall that the two polytopes are said to be \textit{projectively equivalent} if they can be obtained from each other by a projective transformation. We will need the following fact.

\begin{prop}\cite[Lemma 20]{PP}\label{lemprojeq}
If $P,P'\subset\R^d$ are projectively equivalent polytopes, then $\ic(P)=\ic(P')$.
\end{prop}

We also need a sufficient condition for polytopes to be projectively equivalent. Let a polytope $P$ (with $v$ vertices and $f$ facets) be defined as the set of all points $x\in\R^n$ satisfying the conditions $c_i(x)\geq \beta_i$ and $c_j(x)=\beta_j$, for $i\in\{1,\dots,f\}$ and $j\in\{f+1,\dots,q\}$, where $c_1,\dots,c_q$ are linear functionals on $\R^n$. A slack matrix $S=S(P)$ of $P$ is an $f$-by-$v$ matrix satisfying $S_{it}=c_i(p_t)-\beta_i$, where $p_1,\dots,p_v$ denote the vertices of $P$, and we note that $S$ is nonnegative. We remark in passing that $\operatorname{rank}(S)=\dim P$, and the seminal result by Yannakakis~\cite{Yan} states that $\operatorname{rank}_+(S)=\ic(P)$. We are not going to use these results in what follows, but we need the following characterization. We say that matrices $S_1$, $S_2$ coincide up to \textit{scaling} if there are diagonal invertible matrices $D_1$, $D_2$ such that $S_1=D_1S_2D_2$.

\begin{prop}\cite[Corollary 1.5]{GPRT}\label{propslackeq}
If slack matrices of polytopes $P_1$, $P_2$ coincide up to scaling, then $P_1$ and $P_2$ are projectively equivalent.
\end{prop}

\section{The proof}\label{secpr}

Recall that we assume $n\geqslant 3$. In this section, we use the letters $i,j,k$ as indexes of coordinates of $n$-vectors and entries of $n\times n$ matrices, and we assume that these indexes belong to $\mathbb{Z}/n\mathbb{Z}$. In other words, we will assume that $i+1$ stands for $1$ if $i=n$.

Let $\alpha=(\alpha_1,\ldots,\alpha_n)$ be a real vector whose coordinates are algebraically independent over $\mathbb{Q}$. The $n\times n$ matrix $D$ is defined as $D_{ij}=(\alpha_i-\alpha_j)^2$, and our goal is to prove that $\operatorname{rank}_+(D)\geqslant2\sqrt{n}-2$. Let us note that a permutation of $\alpha$ leads to the corresponding permutation of rows and columns of $D$, which does not change nonnegative rank. Therefore, we can assume that the sequence $\alpha$ is increasing. Also, let us define $d_i$ as the sum of the entries in the $i$th column of $D$; we define $D'$ as the matrix obtained from $D$ by dividing every entry $D_{ij}$ by $d_j$. Clearly, the matrix $D'$ is column-stochastic and satisfies $\operatorname{rank}_+(D)=\operatorname{rank}_+(D')$. Let us begin with the computation of the polytope $\pout(D')$ mentioned in Proposition~\ref{proppinpout}.

\begin{claim}\label{cla0}
Let $u_k\in\R^n$ be the vector whose $i$th coordinate equals $(\alpha_i-\alpha_k)(\alpha_i-\alpha_{k+1})$. We have $\operatorname{rank}(D)=3$ and $u_k\in\col(D)$.
\end{claim}

\begin{proof}
One can check that the $u_k$'s and $\col(D)$ are spanned by vectors $(1,\ldots,1)^\top$, $(\alpha_1,\ldots,\alpha_n)^\top$, $(\alpha_1^2,\ldots,\alpha_n^2)^\top$. 
\end{proof}

\begin{claim}\label{cla1}
The polytope $\pout(D')$ is an $n$-gon. The vertex $v_k$ of $\pout(D')$ is $s_k^{-1}u_k$, where $s_k$ is the sum of the coordinates of the vector $u_k$ as in Claim~\ref{cla0}.
\end{claim}

\begin{proof}
Since $\operatorname{rank}(D)=3$, the affine subspace $H=\col(D')\cap\{x_1+\ldots+x_n=1\}$ has dimension $2$, and we get $\dim\pout(D')=2$. Therefore, $\pout(D')$ is a polygon, and every edge of it comes as an intersection of $H$ and a facet of $\Delta_n$. We see that $\pout(D')$ has at most $n$ edges and, therefore, at most $n$ vertices. The vertices of $\pout(D')$ are intersections of $H$ with ridges of $\Delta_n$; in other words, the vertices are the nonnegative vectors in $H$ that have two zero coordinates. By Claim~\ref{cla0} the vectors in the assertion satisfy these properties; so we have identified all the $n$ vertices of $\pout(D')$.
\end{proof}

\begin{claim}\label{cla2}
A slack matrix of $\pout(D')$ is $(v_1|\ldots|v_n)$, where the $v_k$'s are as in Claim~\ref{cla1}.
\end{claim}

\begin{proof}
The polygon $\pout(D')$ is defined by the equality $x_1+\ldots+x_n=1$, the equalities defining $\col(D)$, and the inequalities $x_i\geqslant0$. Therefore, $x_1\geqslant0,\ldots,x_n\geqslant0$ are facet defining inequalities of $\pout(D')$, and the $(i,j)$th entry of the slack matrix equals the $i$th coordinate of the $j$th vertex.
\end{proof}

\begin{claim}\label{cla3}
Every edge of $\pout(D')$ contains a vertex of $\pin(D')$.
\end{claim}

\begin{proof}
Note that the $i$th column of $D'$ is a vertex of $\pin(D')$ and has a zero at the $i$th coordinate. Therefore, this column belongs to the convex hull of those vertices of $\pout(D')$ that have zeros at their $i$th coordinates. By Claim~\ref{cla1}, there are only two such vertices, $v_i$ and $v_{i-1}$, and their convex hull is the edge connecting them.
\end{proof}

\begin{claim}\label{cla4}
Let $P$ be a polygon satisfying $\pin(D')\subset P\subset \pout(D')$. Then any edge of $\pout(D')$ contains some vertex of $P$.
\end{claim}

\begin{proof}
Follows directly from Claim~\ref{cla3}. 
\end{proof}

\begin{claim}\label{cla5}
We define the points $w_k=(w_{k1},w_{k2})\in\R^2$, where $k\in\{1,\ldots,n\}$ and
$$w_{k1}=\frac{1}{\alpha_k}+\frac{1}{\alpha_{k+1}}+\frac{1}{\alpha_k\alpha_{k+1}},\,\,\,\,\,\,
w_{k2}=-\frac{1}{\alpha_k}-\frac{1}{\alpha_{k+1}}+\frac{1}{\alpha_k\alpha_{k+1}}.$$
The polygon $W=\operatorname{conv}\{w_1,\ldots,w_n\}$ is projectively equivalent to $\pout(D')$.
\end{claim}

\begin{proof}
In view of Proposition~\ref{propslackeq}, it suffices to proof that the slack matrices of $W$ and $\pout(D')$ can be obtained from each other by the scaling of rows and columns. The slack matrix of $W$ can be obtained as the matrix $S$ in which the $(i,k)$th entry is the oriented volume of the triangle with vertices $w_{i-1}$, $w_i$, $w_{k}$. That is, we have 
$$S_{ik}=\det \begin{pmatrix}w_{i-1,1}&w_{i-1,2}&1\\
w_{i1}&w_{i2}&1\\
w_{k1}&w_{k2}&1\end{pmatrix},$$
and the straightforward checking shows that
$$S_{ik}=\frac{2(\alpha_{i-1}-\alpha_{i+1})}{\alpha_{i-1}\alpha_i^2\alpha_{i+1}}\cdot\frac{1}{\alpha_{k}\alpha_{k+1}}\cdot(\alpha_i-\alpha_k)(\alpha_i-\alpha_{k+1}).$$ Here, the first multiplier is independent of a column index, the second multiplier is independent of a row index, so we see that the matrix $S$ can be obtained by scaling from the matrix $S'$ defined as $S_{ik}'=(\alpha_i-\alpha_k)(\alpha_i-\alpha_{k+1})$. The matrix $S'$ coincides with the matrix as in Claim~\ref{cla2} up to the scaling of columns, so we are done.
\end{proof}

\begin{claim}\label{cla6}
Let $h_k$ be a point on the straight line connecting the points $w_{k-1}$ and $w_k$ as in Claim~\ref{cla5}. Then $\alpha_k$ is algebraic in the coordinates of $h_k$.
\end{claim}

\begin{proof}
The coordinates of $h_k$ are $\lambda w_{k-1,1}+\mu w_{k1}$ and $\lambda w_{k-1,2}+\mu w_{k2}$, for some $\lambda,\mu\in\R$ satisfying $\lambda+\mu=1$. The half-sum and half-difference of these coordinates are equal, respectively, to
$$\sigma_1=\frac{1}{\alpha_k}\cdot\left(\frac{\lambda}{\alpha_{k-1}}+\frac{\mu}{\alpha_{k+1}}\right),\,\,\,\,\,\,\sigma_2=\frac{\lambda}{\alpha_{k-1}}+\frac{\mu}{\alpha_{k+1}}+\frac{1}{\alpha_k}.$$
By Vieta's formulas, one of the roots of the equation $t^2-\sigma_2t+\sigma_1=0$ equals $1/\alpha_k$ (while the other is $\lambda/\alpha_{k-1}+\mu/\alpha_{k+1}$).
\end{proof}

\begin{claim}\label{cla7}
Let $P$ be a polygon satisfying $\pin(D')\subset P\subset \pout(D')$. Then $\ic(P)\geqslant2\sqrt{n}-2$.
\end{claim}

\begin{proof}
The polytopes $\pout(D')$ and $W$ are projectively equivalent by Claim~\ref{cla5}; let $\pi$ be a projective transformation sending $\pout(D')$ to $W$. Claim~\ref{cla4} shows that every edge of $W$ contains some vertex of $P'=\pi(P)$, and from Claim~\ref{cla6} we get that any $\alpha_k$ is algebraic over $\mathbb{Q}(P')$. Since $\alpha_k$'s are algebraically independent, we get $\trd(P')\geqslant n$, and Theorem~\ref{thrlowtrdeg} implies $\ic(P')\geqslant2\sqrt{n}-2$. Finally, Proposition~\ref{lemprojeq} implies $\ic(P)=\ic(P')$, and we get the desired result.
\end{proof}

In view of Proposition~\ref{proppinpout}, Claim~\ref{cla7} completes the proof of Theorem~\ref{thrdistlower}. Therefore, we get the lower bound for $\operatorname{rank}_{+}(D)$, which allows us to prove all the results announced in Section~1. We note that this bound is still quite far from the best known upper bound, which is $O(n/\ln^{\circ6}n)$, see~\cite{mysublin}. (Here $\ln^{\circ6}$ denotes the sixth iteration of the logarithm.) Proving that there are $n\times n$ distance matrices $D$ satisfying  $\operatorname{rank}_{+}(D)\in\omega(\sqrt{n})$ seems to require an essentially new technique, and our approach does not seem to lead to such an improvement. In particular, it is known~\cite{mysublin} that there are generic $n$-gons $P$ satisfying $\ic(P)\in O(\sqrt{n})$, which means that Theorem~\ref{thrlowtrdeg} cannot be improved substantially.

\smallskip

I would like to thank Troy Lee for pointing my attention to this problem and helpful comments. 

\end{document}